\documentclass[12pt]{amsart}
\usepackage{amsmath,amsthm,amsfonts,amssymb,latexsym,enumerate}

\begin{document}

\theoremstyle{plain}

\newtheorem{thm}{Theorem}[section]
\newtheorem{lem}[thm]{Lemma}
\newtheorem{conj}[thm]{Conjecture}
\newtheorem{pro}[thm]{Proposition}
\newtheorem{cor}[thm]{Corollary}
\newtheorem{que}[thm]{Question}
\newtheorem{rem}[thm]{Remark}
\newtheorem{defi}[thm]{Definition}

\newtheorem*{thmA}{Theorem A}
\newtheorem*{thmB}{Theorem B}
\newtheorem*{thmC}{Theorem C}

\newtheorem*{thmAcl}{Main Theorem$^{*}$}
\newtheorem*{thmBcl}{Theorem B$^{*}$}

\numberwithin{equation}{section}

\newcommand{\Maxn}{\operatorname{Max_{\textbf{N}}}}
\newcommand{\Syl}{\operatorname{Syl}}
\newcommand{\dl}{\operatorname{dl}}
\newcommand{\Con}{\operatorname{Con}}
\newcommand{\cl}{\operatorname{cl}}
\newcommand{\Stab}{\operatorname{Stab}}
\newcommand{\Aut}{\operatorname{Aut}}
\newcommand{\Ker}{\operatorname{Ker}}
\newcommand{\fl}{\operatorname{fl}}
\newcommand{\Irr}{\operatorname{Irr}}
\newcommand{\SL}{\operatorname{SL}}
\newcommand{\GL}{\operatorname{GL}}
\newcommand{\SU}{\operatorname{SU}}
\newcommand{\GU}{\operatorname{GU}}
\newcommand{\Sp}{\operatorname{Sp}}
\newcommand{\Spin}{\operatorname{Spin}}
\newcommand{\PGL}{\operatorname{PGL}}
\newcommand{\PGU}{\operatorname{PGU}}
\newcommand{\FF}{\mathbb{F}}
\newcommand{\NN}{\mathbb{N}}
\newcommand{\N}{\mathbf{N}}
\newcommand{\C}{\mathbf{C}}
\newcommand{\OO}{\mathbf{O}}
\newcommand{\F}{\mathbf{F}}

\renewcommand{\labelenumi}{\upshape (\roman{enumi})}

\newcommand{\PSL}{\operatorname{PSL}}
\newcommand{\PSU}{\operatorname{PSU}}

\providecommand{\St}{\mathsf{St}}
\providecommand{\E}{\mathrm{E}}
\providecommand{\PSp}{\mathrm{PSp}}
\providecommand{\Sp}{\mathrm{Sp}}
\providecommand{\SO}{\mathrm{SO}}
\providecommand{\Irrr}{\mathrm{Irr_{rv}}}
\providecommand{\re}{\mathrm{Re}}

\def\irrp#1{{\rm Irr}_{p'}(#1)}

\def\Z{{\mathbb Z}}
\def\C{{\mathbb C}}
\def\Q{{\mathbb Q}}
\def\irr#1{{\rm Irr}(#1)}
\def\ibr#1{{\rm IBr}(#1)}
\def\ipi#1{{\bf I}_\pi(#1)}
\def\ibrt#1{{\rm IBr}_{2'}(#1)}

\def\irrv#1{{\rm Irr}_{\rm rv}(#1)}
\def \c#1{{\mathcal #1}}
\def\cent#1#2{{\bf C}_{#1}(#2)}
\def\syl#1#2{{\rm Syl}_#1(#2)}
\def\nor{\triangleleft\,}
\def\oh#1#2{{\bf O}_{#1}(#2)}
\def\Oh#1#2{{\bf O}^{#1}(#2)}
\def\zent#1{{\bf Z}(#1)}
\def\det#1{{\rm det}(#1)}
\def\ker#1{{\rm ker}(#1)}
\def\norm#1#2{{\bf N}_{#1}(#2)}
\def\alt#1{{\rm Alt}(#1)}
\def\iitem#1{\goodbreak\par\noindent{\bf #1}}
   \def \mod#1{\, {\rm mod} \, #1 \, }
\def\sbs{\subseteq}
\def\irrtb#1{{\rm Irr}_{p'}({\rm B_0}(#1))}

\newcommand{\Out}{{\mathrm {Out}}}
\newcommand{\Mult}{{\mathrm {Mult}}}
\newcommand{\Inn}{{\mathrm {Inn}}}
\newcommand{\IBR}{{\mathrm {IBr}}}
\newcommand{\IBRL}{{\mathrm {IBr}}_{\ell}}
\newcommand{\IBRP}{{\mathrm {IBr}}_{p}}
\newcommand{\ord}{{\mathrm {ord}}}
\def\id{\mathop{\mathrm{ id}}\nolimits}
\renewcommand{\Im}{{\mathrm {Im}}}
\newcommand{\Ind}{{\mathrm {Ind}}}
\newcommand{\diag}{{\mathrm {diag}}}
\newcommand{\soc}{{\mathrm {soc}}}
\newcommand{\End}{{\mathrm {End}}}
\newcommand{\sol}{{\mathrm {sol}}}
\newcommand{\Hom}{{\mathrm {Hom}}}
\newcommand{\Mor}{{\mathrm {Mor}}}
\newcommand{\Mat}{{\mathrm {Mat}}}
\def\rank{\mathop{\mathrm{ rank}}\nolimits}
\newcommand{\Tr}{{\mathrm {Tr}}}
\newcommand{\tr}{{\mathrm {tr}}}
\newcommand{\Gal}{{\it Gal}}
\newcommand{\Spec}{{\mathrm {Spec}}}
\newcommand{\ad}{{\mathrm {ad}}}
\newcommand{\Sym}{{\mathrm {Sym}}}
\newcommand{\Char}{{\mathrm {char}}}
\newcommand{\pr}{{\mathrm {pr}}}
\newcommand{\rad}{{\mathrm {rad}}}
\newcommand{\abel}{{\mathrm {abel}}}
\newcommand{\codim}{{\mathrm {codim}}}
\newcommand{\ind}{{\mathrm {ind}}}
\newcommand{\Res}{{\mathrm {Res}}}
\newcommand{\Ann}{{\mathrm {Ann}}}
\newcommand{\Ext}{{\mathrm {Ext}}}
\newcommand{\Alt}{{\mathrm {Alt}}}
\newcommand{\AAA}{{\sf A}}
\newcommand{\SSS}{{\sf S}}
\newcommand{\CC}{{\mathbb C}}
\newcommand{\CB}{{\mathbf C}}
\newcommand{\RR}{{\mathbb R}}
\newcommand{\QQ}{{\mathbb Q}}
\newcommand{\ZZ}{{\mathbb Z}}
\newcommand{\NB}{{\mathbf N}}
\newcommand{\ZB}{{\mathbf Z}}
\newcommand{\OB}{{\mathbf O}}
\newcommand{\EE}{{\mathbb E}}
\newcommand{\PP}{{\mathbb P}}
\newcommand{\GC}{{\mathcal G}}
\newcommand{\HC}{{\mathcal H}}
\newcommand{\AC}{{\mathcal A}}
\newcommand{\SC}{{\mathcal S}}
\newcommand{\TC}{{\mathcal T}}
\newcommand{\CL}{{\mathcal C}}
\newcommand{\EC}{{\mathcal E}}
\newcommand{\GCD}{\GC^{*}}
\newcommand{\TCD}{\TC^{*}}
\newcommand{\FD}{F^{*}}
\newcommand{\GD}{G^{*}}
\newcommand{\HD}{H^{*}}
\newcommand{\GCF}{\GC^{F}}
\newcommand{\TCF}{\TC^{F}}
\newcommand{\PCF}{\PC^{F}}
\newcommand{\GCDF}{(\GC^{*})^{F^{*}}}
\newcommand{\RGTT}{R^{\GC}_{\TC}(\theta)}
\newcommand{\RGTA}{R^{\GC}_{\TC}(1)}
\newcommand{\Om}{\Omega}
\newcommand{\eps}{\epsilon}
\newcommand{\varep}{\varepsilon}
\newcommand{\al}{\alpha}
\newcommand{\chis}{\chi_{s}}
\newcommand{\sigmad}{\sigma^{*}}
\newcommand{\PA}{\boldsymbol{\alpha}}
\newcommand{\gam}{\gamma}
\newcommand{\lam}{\lambda}
\newcommand{\la}{\langle}
\newcommand{\ra}{\rangle}
\newcommand{\hs}{\hat{s}}
\newcommand{\htt}{\hat{t}}
\newcommand{\tn}{\hspace{0.5mm}^{t}\hspace*{-0.2mm}}
\newcommand{\ta}{\hspace{0.5mm}^{2}\hspace*{-0.2mm}}
\newcommand{\tb}{\hspace{0.5mm}^{3}\hspace*{-0.2mm}}
\def\skipa{\vspace{-1.5mm} & \vspace{-1.5mm} & \vspace{-1.5mm}\\}
\newcommand{\tw}[1]{{}^#1\!}
\renewcommand{\mod}{\bmod \,}

\marginparsep-0.5cm

\renewcommand{\thefootnote}{\fnsymbol{footnote}}
\footnotesep6.5pt

\title{Weights and Nilpotent Subgroups}
\author{Gabriel Navarro}
\address{Department of Mathematics,   Universitat de Val\`encia, 46100 Burjassot,
Val\`encia, Spain}
\email{gabriel.navarro@uv.es}
\author{Benjamin Sambale}
\address{Institut f\"ur Mathematik, Friedrich-Schiller-Universit\"at Jena, 07737 Jena, Germany}
\email{benjamin.sambale@uni-jena.de}

\thanks{The research of the first  author is supported by 
MTM2016-76196-P and Prometeo/Generalitat Valenciana.
The second author thanks the German Research
Foundation (projects SA 2864/1-1 and SA 2864/3-1).}



\begin{abstract}
In a finite group $G$, we consider nilpotent weights,
and prove a $\pi$-version of the Alperin Weight Conjecture
for certain $\pi$-separable groups. This widely generalizes an earlier result by I. M. Isaacs and the first author.
\end{abstract}

\maketitle

\section{Introduction} 
\medskip
Let $G$ be a finite group and let $p$ be a prime.
The celebrated Alperin Weight Conjecture asserts that
the number of conjugacy classes of $G$ consisting of
elements of order not divisible by $p$ is exactly the number
of $G$-conjugacy classes of $p$-weights. Recall that a {\bf $p$-weight}
is a pair $(Q,\gamma)$, where $Q$ is a $p$-subgroup of $G$
and $\gamma \in \irr{\norm GQ/Q}$ is an irreducible complex
character with $p$-defect zero (that is,
such that the $p$-part $\gamma(1)_p=|\norm GQ/Q|_p$).

In the main result of this paper, we replace $p$ by a set of primes $\pi$ as follows:

\begin{thmA}
Let $G$ be a $\pi$-separable group with a solvable Hall $\pi$-subgroup. Then the number of conjugacy classes of
$\pi'$-elements of $G$ is the number of $G$-conjugacy classes of pairs
$(Q, \gamma)$, where $Q$ is a nilpotent $\pi$-subgroup of $G$
and $\gamma \in \irr{\norm GQ/Q}$ has $p$-defect zero for every $p\in\pi$.
\end{thmA}
Recall that a finite group is called \textbf{$\pi$-separable} if all its composition factors are $\pi$-groups or $\pi'$-groups.
Let us restate Theorem A in the (presumably trivial) case where $G$ itself is a (solvable) $\pi$-group.
In this case, there is only one conjugacy class of $\pi'$-elements of $G$. On the other hand,
if $Q$ is a nilpotent subgroup of $G$, then  $\gamma \in \irr{\norm GQ/Q}$
has $p$-defect zero for every $p\in\pi$ if and only if
$\norm GQ=Q$. Amazingly enough, there is only one conjugacy class of self-normalizing
nilpotent subgroups: the Carter subgroups of $G$ (see p. 281 in \cite{R}).

\medskip
Of course, if $\pi=\{p\}$, then Theorem A is the $p$-solvable case
of the Alperin Weight Conjecture (AWC). As a matter of fact, AWC was proven for $\pi$-separable groups with a nilpotent Hall $\pi$-subgroup by Isaacs and the first author~\cite{IN}. Now we realize that the nilpotency hypothesis can be dropped if one counts nilpotent weights instead.
The solvability hypothesis is still needed, as shown by $G={\sf A}_5$ and $\pi=\{2,3,5\}$.

There is a price to pay, however. The proof in \cite{IN} relied on the so called Okuyama--Wajima argument, a 
definitely  non-trivial but accessible
tool on extensions of Glauberman correspondents.
In order to prove Theorem~A, however, we shall need to appeal to a deeper theorem of
 Dade and Puig 
 (which uses  Dade's classification of the endo-permutation modules). 
 
 \medskip
As it is often the case in a ``solvable'' framework, the equality of cardinalities in Theorem~A has a hidden structure which we are going to explain now. 
For sake of convenience we interchange from now on the roles of $\pi$ and $\pi'$ (of course, $\pi$-separable is equivalent to $\pi'$-separable).
 Recall that in a $\pi$-separable group $G$, the set $\ipi G$ of 
   irreducible $\pi$-partial characters of $G$ is the exact $\pi$-version 
 (when $\pi$ is the complement of a prime $p$) of the irreducible Brauer characters
 $\ibr G$ of a $p$-solvable group (see next section for precise definitions). Each $\varphi \in \ipi G$ has canonically
 associated a $G$-conjugacy class of $\pi'$-subgroups $Q$, which are called the {\bf vertices}
 of $\varphi$. If $\ipi{G|Q}$ is the set of irreducible $\pi$-partial characters with vertex $Q$, unless
 $\pi=p'$, it is not in general true that  $|\ipi{G|Q}|= |\ipi{\norm GQ|Q}|$.
Instead we will prove the following theorem.
 
 \begin{thmB}
 Suppose that $G$ 
is $\pi$-separable with a solvable Hall $\pi$-complement.
Let $R$ be a nilpotent $\pi'$-subgroup 
of $G$ and let $\mathcal Q$ be the set of $\pi'$-subgroups $Q$ of $G$
such that $R$
is a Carter subgroup 
of $Q$.
Then  
$$\Bigl|\bigcup_{Q \in {\mathcal Q}}\ipi{G|Q}\Bigr|=|\ipi{\norm GR|R}| \,.$$
 \end{thmB}
 
Since $|\ipi{\norm GR|R}|$ is just the number of $\pi'$-weights with first component $R$ (see Lemma~6.28 of \cite{I2}), Theorem~B implies Theorem~A. 
\medskip

 As happens in the classical case where $\pi=p'$, and following the ideas
 of Dade, Kn\"orr and Robinson, one can define
 chains of $\pi'$-subgroups and relate them with
 $\pi$-defect of characters. This shall be explored elsewhere.
Similarly, one can attach every weight to a $\pi'$-block $B$ of $G$ by using Slattery's theory~\cite{Slattery}. In this setting we expect that the number of $\pi$-partial characters belonging to $B$ equals the number of nilpotent weights attached to $B$. 

The groups described in Theorem~A are sometimes called $\pi$-solvable. We did not find a counterexample in the wider class of so-called $\pi$-selected groups. Here, $\pi$-\textbf{selected} means that the order of every composition factor is divisible by at most one prime in $\pi$. P. Hall~\cite{H} has shown that these groups still have solvable Hall $\pi$-subgroups. Since every finite group is $p$-selected for every prime $p$, this version of the conjecture includes AWC in full generality.
 
 \medskip
 
Unfortunately, Theorem~A does not hold for arbitrary groups even if 
they possess nilpotent Hall $\pi$-subgroups. It is not so easy to find a counterexample, though. The fourth Janko group $G=J_4$ has a cyclic Hall $\pi$-subgroup of order $35$ (that is, $\pi=\{5,7\}$). The normalizers of the non-trivial $\pi$-weights are contained in a maximal subgroup $M$ of type $2^{3+12}.(S_5\times L_3(2))$. However, $l(G)-k_0(G)=25\ne 30=l(M)-k_0(M)$, where $l(G)$ denotes the number of $\pi'$-conjugacy classes and $k_0(G)$ is the number of $\pi$-defect zero characters of $G$.
(The fact that $J_4$ was a counterexample for the $\pi$-version of the McKay conjecture
for groups with a nilpotent Hall $\pi$-subgroup was  noticed by Pham H. Tiep and the first author.)
 
We take the opportunity to thank the developers of \cite{GAP}. Without their tremendous work the present paper probably would  not exist.

 \medskip
The paper is organized as follows: In the next section we review $\pi$-partial characters which were introduced by Isaacs. In Section~3 we present two general lemmas on characters in $\pi$-separable groups. Afterwards we prove Theorem~B. In the final section we construct a natural bijection explaining Theorem~B in the presence of a normal Hall $\pi$-subgroup.

\section{Review of $\pi$-theory}

Isaacs' $\pi$-theory is the $\pi$-version in $\pi$-separable groups
of the $p$-modular representation theory for
$p$-solvable groups. When $\pi=p'$, the complement of 
 a prime, then  $\ipi G=\ibr G$ and we recover most of the well-known classical
 results. In what follows $G$ is a finite
$\pi$-separable group, where $\pi$ is a set of primes. All the references for $\pi$-theory can now be found together in Isaacs' recent book~\cite{I2}. For the reader's convenience, we review some of
the main features.  If $n$ is a natural number and $p$ is a prime,
recall that $n_p$ is the largest power of $p$ dividing $n$. If $\pi$ is a set of primes,
then $n_\pi=\prod_{p \in \pi} n_p$. The number $n$ is a $\pi$-number if $n=n_\pi$.

\medskip

If $G$ is a $\pi$-separable group, then $G^0$ is the set of elements of $G$
whose order is a $\pi$-number. A $\pi$-partial character of $G$ is the restriction of
a complex character of $G$ to $G^0$. A $\pi$-partial character is {\bf irreducible}
if it is not the sum of two $\pi$-partial characters. We write $\ipi G$ for the set
of irreducible $\pi$-partial characters of $G$.  Notice that if $\mu \in \ipi G$
by definition there exists $\chi \in \irr G$ such that $\chi^0=\mu$, where $\chi^0$
denotes the restriction of $\chi$ to the $\pi$-elements of $G$. 
Also, it is clear by the definition, that every $\pi$-partial character
is a sum of irreducible $\pi$-partial characters. Notice that if $G$ is a $\pi$-group,
then $\ipi G=\irr G$.

\medskip

\begin{thm}[Isaacs]\label{ipi}
Let $G$ be a finite $\pi$-separable group. Then $\ipi G$ is a basis of the
space of class functions defined on $G^0$. In particular, $|\ipi G|$ is the number
of conjugacy classes of $\pi$-elements of $G$.
\end{thm}

\begin{proof}
This is Theorem 3.3 of \cite{I2}. 
\end{proof}

\medskip
We can induce and restrict $\pi$-partial characters in a natural way.
If $H$ is a subgroup of $G$ and $\varphi \in \ipi G$, then 
$\varphi_H=\sum_{\mu \in \ipi H} a_\mu \mu$
for some uniquely defined nonnegative integers $a_\mu$. We write $\ipi{G|\mu}$
to denote the set of  $\varphi \in \ipi G$
such that $a_\mu \ne 0$.

\medskip

A   non-trivial result is that Clifford's theory holds for 
$\pi$-partial characters. If $N\nor G$, it is then clear that  $G$ naturally acts on $\ipi N$
by   conjugation.

\begin{thm}[Isaacs]\label{ipicliff}
Suppose that $G$ is $\pi$-separable and $N \nor G$.

\begin{enumerate}[(a)]
\item
If $\varphi \in \ipi G$, then $\varphi_N=e(\theta_1 + \cdots + \theta_t)$,
where $\theta_1,\ldots,  \theta_t$ are all the $G$-conjugates of
some $\theta \in \ipi N$.

\item
If $\theta \in \ipi N$ and $T=G_\theta$ is the stabilizer of $\theta$
in $G$, then induction defines a bijection $\ipi{T|\theta} \rightarrow \ipi{G|\theta}$.
\end{enumerate}
\end{thm}

\begin{proof}
See Corollary~5.7 and Theorem~5.11 of \cite{I2}. 
\end{proof}

In part (b) of Theorem \ref{ipicliff}, if $\mu^G=\varphi$, where $\mu \in \ipi{T|\theta}$,
then $\mu$ is
called the {\bf Clifford correspondent} of $\varphi$ over $\theta$, and sometimes
it is written $\mu=\varphi_\theta$.

It is not a triviality to define vertices for $\pi$-partial characters
(a concept that in classical modular representation theory has little to do
with character theory). This was first accomplished in \cite{IN} (generalizing a result of Huppert on Brauer characters of $p$-solvable groups).

\begin{thm}\label{vertices}
Suppose that $G$ is $\pi$-separable, and let $\varphi \in \ipi G$.
Then there exist a subgroup $U$ of $G$ and $\alpha \in \ipi U$
of $\pi$-degree such that $\alpha^G=\varphi$. Furthermore,
if $Q$ is a Hall $\pi$-complement of $U$, then the $G$-conjugacy
class of $Q$ is uniquely determined by $\varphi$.
\end{thm}

\begin{proof}
This is Theorem  5.17 of \cite{I2}. 
\end{proof}

The uniquely defined $G$-class of $\pi'$-subgroups $Q$
associated to $\varphi$ by Theorem \ref{vertices} is 
called the set of {\bf vertices} of $\varphi$.
If $Q$ is a $\pi'$-subgroup of $G$, then we write
$\ipi{G|Q}$ to denote the set of $\varphi \in \ipi G$ which have $Q$ as a vertex.
By definition, notice in this case that
$$\varphi(1)_{\pi'}=|G:Q|_{\pi'}\, .$$

Our last important ingredient is the Glauberman correspondence.

\begin{thm}[Glauberman]
Let $S$ be a finite solvable group acting via automorphisms on a finite group $G$ such that $(|S|,|G|)=1$. Then there exists a canonical bijection, called the \textbf{$S$-Glauberman correspondence}, 
\[\Irr_S(G)\to\irr C, \qquad\chi\mapsto\chi^*,\] 
where $\Irr_S(G)$ is the set of $S$-invariant irreducible characters of $G$ and $C=\cent GS$. Here, $\chi^*$ is a constituent of the restriction $\chi_C$.
Also, if $T \nor S$, then the $T$-Glauberman correspondence is an isomorphism of $S$-sets.  
\end{thm}
\begin{proof}
See Theorem~13.1 of \cite{I1}. 
\end{proof}

\section{Preliminaries}
If $G$ is a finite group, $\pi$ is a set of primes,
and $\chi \in \irr G$, then we say that
$\chi$ has {\bf $\pi$-defect zero} if $\chi(1)_{\pi}=|G|_{\pi}$.

\begin{lem}\label{pirad}
If $\chi \in \irr G$ has $\pi$-defect zero, then $\oh{\pi} G=1$.
\end{lem}

\begin{proof}
Let $N \nor G$, and let $\theta \in \irr N$ be under $\chi$.
Then we have that $\theta(1)$ divides $|N|$ and $\chi(1)/\theta(1)$
divides $|G:N|$ by Corollary 11.29 of \cite{I1}.
Thus $\chi(1)_\pi=|G|_\pi$ if and only if $\theta(1)_\pi=|N|_\pi$ and 
$(\chi(1)/\theta(1))_\pi=|G:N|_\pi$. The result is now clear
applying this to $N=\oh \pi G$.
\end{proof}

We shall use the following notation.
Suppose $G$ is $\pi$-separable, 
 $N \nor G$, $\tau \in \ipi N$ and $Q$ is a $\pi'$-subgroup of $G$.
 Then
 $$\ipi{G|Q,\tau}=\ipi{G|Q} \cap \ipi{G|\tau} \, .$$

\begin{lem}\label{gvert}
Suppose $G$ is $\pi$-separable
and that $N \nor G$. Let $Q$ be a $\pi'$-subgroup of $G$.
\begin{enumerate}[(a)]
\item
Suppose that $\mu \in \ipi{G|Q}$. Then there is a unique
$\norm GQ$-orbit of  $\tau \in \ipi N$ such that
$\mu_\tau \in \ipi{G_\tau|Q}$, where $\mu_\tau$ is the Clifford
correspondent of $\mu$ over $\tau$. Every such $\tau$ is
$Q$-invariant.

\item
Suppose that $\tau \in \ipi N$ is $Q$-invariant.
Let $\c U$ be a complete set of representatives
of the $G_\tau$-orbits on the set $\{Q^g \, |\,   g \in G, Q^g \sbs G_\tau\}$.
Then 
$$|\ipi{G|Q,\tau}|=\sum_{U \in {\mathcal U}} |\ipi{G_\tau|U,\tau}|\, .$$
Thus, if $G_\tau\norm GQ=G$, then  
$$|\ipi{G|Q,\tau}|=|\ipi{G_\tau|Q,\tau}|\, .$$
\end{enumerate}
\end{lem}

\begin{proof}\begin{enumerate}[(a)]
\item Let $\nu \in \ipi N$ be under $\mu$, and let $\mu_\nu \in \ipi{G_\nu|\nu}$
be the Clifford correspondent of $\mu$ over $\nu$. 
If $R$ is a vertex of $\mu_\nu$, then $R$ is a vertex of $\mu$, by Theorem~\ref{vertices}.
Therefore $R=Q^g$ for some $g \in G$. If $\tau=\nu^{g^{-1}}$, then we have that
$\mu_\tau$ has vertex $Q$. Suppose now that
 $\rho \in \ipi N$ is under $\mu$ such that
$\mu_\rho$ has vertex $Q$. By Theorem~\ref{ipicliff}(a), there exists $g \in G$ such that $\tau^g=\rho$. Thus $Q^g$ is a vertex of $\mu_\rho$.
Then there is $x \in G_\rho$ such that $Q^{gx}=Q$. Since $\tau^{gx}=\rho$,
the proof of part (a) is complete.

\item We have that induction defines a bijection $\ipi{G_\tau|\tau} \rightarrow \ipi{G|\tau}$.
Notice that 
$$\bigcup_{U \in {\c U}} \ipi{G_\tau|U}$$
is a disjoint union.
It suffices to observe, again,  that if $\xi \in \ipi{G_\tau|\tau}$ has vertex $U$, then
$\xi^G$ has vertex $U$.\qedhere 
\end{enumerate}
\end{proof}

\section{Proofs}
The deep part in our proofs comes from the following result.

\begin{thm}\label{deep}
Suppose that $L$ is
a normal $\pi$-subgroup of $G$, $Q$ is
a solvable $\pi'$-subgroup of $G$ such that $LQ \nor G$.
Suppose that $M \sbs \zent G$ is contained in $L$
and that $\varphi \in \irr M$. 
Then $|\ipi{G|Q, \varphi}|=|\ipi{\norm GQ|Q, \varphi}|$.
\end{thm}

\medskip

\begin{proof}
Let $\c A$ be a complete set of
representatives of $\norm GQ$-orbits on ${\rm Irr}_Q(L|\varphi)$,
the $Q$-invariant members of $\irr{L|\varphi}$.
Using Lemma \ref{gvert},
we have that  $$\ipi{G|Q, \varphi}=\bigcup_{\tau \in {\c A}} \ipi{G|Q,\tau}$$
is a disjoint union.   Let $\c A^*$ be the set of the $Q$-Glauberman correspondents
 of the elements of $\c A$. Notice that $\c A^*$
 is a complete set of
representatives of $\norm GQ$-orbits on $\irr{\cent LQ|\varphi}$. Moreover, $\norm {G_\tau}Q=\norm {G_{\tau^*}}Q$.
Then, as before,
 $$\ipi{\norm GQ|Q, \varphi}=\bigcup_{\tau \in {\c A}} \ipi{\norm GQ|Q,\tau^*}$$
is a disjoint union.
Thus
 $$|\ipi{\norm GQ|Q,\varphi}|=\sum_{\tau \in {\c A}} |\ipi{\norm {G_\tau}Q|Q, \tau^*}|\,.$$
Thus we need to prove that
 $$|\ipi{G_\tau|Q, \tau}|=|\ipi{\norm {G_\tau}Q|Q, \tau^*}|\,.$$
 We may assume that $\tau$ is $G$-invariant.
 
 Now, since $LQ \nor G$ and $\tau$ is $G$-invariant,
  by Lemma 6.30 of \cite{I2}, we have that $Q$ is contained
  as a normal subgroup in some vertex of $\theta$, whenever $\theta \in \ipi G$
  lies over $\tau$. Therefore 
 $\theta \in \ipi{G|\tau}$ has vertex $Q$ if and only if
  $\theta(1)_{\pi'}=|G:Q|_{\pi'}$. Similarly, $\theta\in\ipi{\norm GQ|\tau^*}$ has vertex $Q$ if any only if $\theta(1)_{\pi'}=|\norm GQ:Q|_{\pi'}=|G:Q|_{\pi'}$ since $G=L\norm GQ$ by the Frattini argument and the Schur--Zassenhaus theorem.
 
 Now we use the Dade--Puig theory on the character theory
 above Glauberman correspondents, which is thoroughly  explained in
 \cite{T}. By  Theorem 6.5 of \cite{T}, in the language of Chapter 11 of \cite{I1}
 (see Definition 11.23 of \cite{I1}), we have
 that the character triples $(G,L,\tau)$ and $(\norm GQ, \cent LQ, \tau^*)$
 are isomorphic.  Write $^*: \irr{G|\tau} \rightarrow \irr{\norm GQ|\tau^*}$
 for the associated bijection of characters.
 By Lemma 6.21 of \cite{I2}, there exists a unique bijection
 $$^{*}: \ipi{G|\tau} \rightarrow \ipi{\norm GQ|\tau^*}$$
 such that if $\chi^0=\phi \in \ipi{G|\tau}$ and $\chi \in \irr G$
 (which necessarily lies over $\tau$), then
 $(\chi^*)^0=\phi^*$. 
 Since $\chi(1)/\tau(1)= \chi^*(1)/\tau^*(1)$
 (by Lemma 11.24 of \cite{I1}), it follows that
 $\chi(1)_{\pi'}=\chi^*(1)_{\pi'}$. 
 We deduce that 
 $$|\ipi{G|Q, \tau}|=|\ipi{\norm {G}Q|Q, \tau^*}|\,,$$
 as desired. 
 \end{proof}

In order to prove Theorem~B, we argue by induction on the index of a normal $\pi$-subgroup $M$ of $G$. Theorem~B follows from the special case $M=1$.

\begin{thm}\label{thm}
Suppose that $G$ is $\pi$-separable with a solvable Hall $\pi$-complement.
Let $R$ be a nilpotent $\pi'$-subgroup 
of $G$. Let $M\nor G$ be a normal
$\pi$-subgroup, and let   $\varphi \in \irr M$ be
$G$-invariant.
Let $\mathcal Q$ be the set of $\pi'$-subgroups $Q$ of $G$
such that $R$
is a Carter subgroup 
of $Q$.
Then  
$$\Bigl|\bigcup_{Q \in {\mathcal Q}}\ipi{G|Q, \varphi}\Bigr|=|\ipi{M\norm GR|R, \varphi}| \,.$$
\end{thm}

\medskip

\begin{proof}
We argue by induction on $|G:M|$. 

By Lemma 3.11 of \cite{I2}, let $(G^*,M^*,\varphi^*)$ be a character triple isomorphic
to $(G,M,\varphi)$, where $M^*$ is a central
$\pi$-subgroup of $G^*$. 
If $Q$ is a $\pi'$-subgroup of $G$, notice that
we can write  $(QM)^*=M^* \times Q^*$, for a unique
 $\pi'$-subgroup $Q^*$ of $G^*$.
 If $R$ is contained in a $\pi'$-subgroup $Q$, then $R$
 is a Carter subgroup of $Q$ if and only if $RM/M$ is a Carter subgroup
 of $QM/M$, using that $Q$ is naturally isomorphic to $QM/M$.
 This happens if and only if $(RM/M)^*$ is a Carter subgroup
 of $(QM/M)^*$, which again happens if and only if $R^*$ is a Carter
 subgroup of $Q^*$. Notice further that if $R$ is a Carter subgroup of $Q$,
 then $R$ is a Carter subgroup of every Hall $\pi$-complement $Q_1$
 of $QM$ that happens to contain $R$ (again using the isomorphism
 between $QM/M$ and $Q$). We easily check now that
 the set of $\pi'$-subgroups of $G^*$ that contain $R^*$ as a Carter subgroup
 is exactly $\mathcal{Q^*}=\{Q^* | Q \in \mathcal Q\}$.

By the Frattini argument and the Schur--Zassenhaus theorem,
 notice that $\norm G{MR}=M\norm GR$.
 By Lemma 6.21 and the proof of Lemma 6.32 of \cite{I2},
 there is a bijection $^*: \ipi{G|\varphi} \rightarrow \ipi{G^*|\varphi^*}$
 such that $\eta$ has vertex $Q$ if and only if $\eta^*$ has vertex $Q^*$.
From all these arguments, it easily follows 
that we may assume that $M$ is central. In particular, $M\le\norm GR$.

Let $K=\oh{\pi'} G$.
Suppose that there exists
some $\mu \in \ipi{G|Q, \varphi}$ for some $Q \in \mathcal Q$.
By
Lemma 6.30 of \cite{I2} (in the notation of that lemma, $K$ is $1$ and $Q$ is $K$), we have that $K$ is contained in $Q$.
Hence, it is no loss if we only consider $Q \in \mathcal Q$
such that $K \sbs Q$. 

Suppose that $\norm KR$ is not contained in $R$.
Then there cannot be weights $(R, \gamma)$, where
$\gamma \in \irr{\norm GR/R}$ has $\pi$-defect zero by Lemma \ref{pirad}.
So the right hand side is zero. 
Suppose that there exists
some $\mu \in \ipi{G|Q, \varphi}$ for some $Q \in \mathcal Q$ (with $K \sbs Q$).
Since $R$ is a Carter subgroup of $Q$, then $R$ is a Carter subgroup of $KR$, and therefore
$\norm KR$ is contained in $R$. Therefore may assume that 
 $\norm KR$ is   contained in $R$.
 We claim that $R$ is a Carter subgroup of $Q$ if and only if  $RK/K$ is a Carter subgroup of $Q/K$.
 One implication is known (see 9.5.3 in \cite{R}). Suppose that $RK/K$ is a Carter subgroup of $Q/K$.
 Since $\norm QR$ normalizes $RK$, it is contained in $RK$.
 Hence $\norm QR=\norm {KR}R=R$, and $R$ is a Carter subgroup of $Q$.
In this situation the Frattini argument yields $\norm GRK=\norm G{RK}$.

 Next, we will replace $G$ by $G/K$. By Lemma 6.31 of \cite{I2} (the roles of $K$ and $M$ are interchanged in that lemma),
\[|\ipi{G|Q,\varphi}|=|\ipi{G/K|Q/K,\hat\varphi},\]
where $\hat\varphi \in \irr{MK/K}$ corresponds to  $\varphi$ via the natural isomorphism.
Similarly,
\[|\ipi{\norm GR|R,\varphi}|=|\ipi{\norm GRK/K|RK/K,\hat\varphi}|=|\ipi{\norm G{RK}/K|RK/K,\hat\varphi}|.\]
Hence, for the remainder of the proof we may assume that $\oh{\pi'} G=K=1$.

Suppose now that $L=\oh \pi G$. 
Let $\c A$ be a complete set of $\norm GR$-representatives
of the $R$-invariant characters in $\irr{L|\varphi}$.
If $L=M$, then $L=G$ by the Hall-Higman Lemma 1.2.3, and $G$ is a $\pi$-group.
In this case, $R=1=Q$,  and there is nothing to prove.
Thus, we may assume that $|G:L|<|G:M|$.

For each $\tau \in \c A$, let ${\c Q}_\tau$ be
the set of $\pi'$-subgroups $Q$ of $G_\tau$ such that
$R$ is a Carter subgroup of $Q$. 
By induction,
$$\Bigl|\bigcup_{Q \in {\mathcal Q}_\tau}\ipi{G_\tau|Q, \tau}\Bigr|=|\ipi{L\norm {G_\tau}R|R, \tau}| \,.$$
Since $L$ is a $\pi$-group and $R$ is a $\pi'$-subgroup,
we have that $L\norm GR=\norm G{LR}$. 
Also, $\norm G{LR}_\tau=L\norm G R_\tau$. By Lemma \ref{gvert},
we  have that 
$$|\ipi{L\norm {G_\tau}R|R, \tau}|=|\ipi{L\norm {G}R|R, \tau}|  \,.$$
Also,
$$|\ipi{L \norm GR|R, \varphi}|= \sum_{\tau \in {\c A}}
|\ipi{L\norm {G}R|R, \tau}| \, ,$$
by using the first paragraph of the proof of Theorem \ref{deep}.
By Theorem \ref{deep},
$$|\ipi{L \norm GR|R, \varphi}|=|\ipi{\norm GR|R, \varphi}| \, .$$
Therefore,
$$\sum_{\tau \in {\c A}} \Bigl|\bigcup_{Q \in {\mathcal Q}_\tau}\ipi{G_\tau|Q, \tau}\Bigr|
=|\ipi{\norm GR|R, \varphi}| \, .$$
We are left to show that
$$\Bigl|\bigcup_{Q \in {\mathcal Q}}\ipi{G|Q, \varphi}\Bigr|=\sum_{\tau \in {\c A}} \Bigl|\bigcup_{Q \in {\mathcal Q}_\tau}\ipi{G_\tau|Q, \tau}\Bigr| \, .$$

Let $\c R$ be a complete set of representatives of $\norm GR$-orbits in $\c Q$, and notice that 
$$\bigcup_{Q \in {\mathcal Q}}\ipi{G|Q, \varphi} = \bigcup_{Q \in {\mathcal R}}\ipi{G|Q, \varphi}$$ is a disjoint union. Indeed, if $\mu \in \ipi{G|Q_1, \varphi} \cap \ipi{G|Q_2, \varphi}$
for $Q_i \in \c Q$, then we have that $Q_1=Q_2^g$ for some $g \in G$ by the uniqueness
of vertices. Hence $R^g$ and $R$ are Carter subgroups of $Q_1$, and therefore
$R^{gx}=R$ for some $x \in Q_1$. It follows that $Q_1=Q_2^{gx}$ are $\norm GR$-conjugate.

Now fix $Q \in \c R$. For each $\mu \in \ipi{G|Q,\varphi}$,
we claim that there is a unique $\tau \in \c A$
such that $\mu_\tau \in \ipi{G_\tau|Q^x, \tau}$, for some $x \in \norm GR$.  We know that 
there is $\nu \in \irr{L|\varphi}$ such that $\mu_\nu \in \ipi{G_\nu|Q, \nu}$ by Lemma \ref{gvert}(a).
Now, $\nu^x=\tau$ for some $x \in \norm GR$ and $\tau \in \c A$, and it follows
that $\mu_\tau \in \ipi{G_\tau|Q^x, \tau}$. Suppose that 
$\mu_\epsilon \in \ipi{G_\epsilon|Q^y, \epsilon}$, for some $y \in \norm GR$ and $\epsilon \in \c A$. Now, $\epsilon=\tau^g$ for some $g \in G$, by Clifford's theorem. Thus $Q^{xgt}=Q^y$ for some $t \in G_\epsilon$, by the uniqueness of vertices. Thus $xgty^{-1} \in \norm GQ$. Since $R$ is a Carter subgroup of $Q$, by the  Frattini argument
we have that $xgty^{-1}=qv$,
where $q \in Q$ and $v \in G$ normalizes $Q$ and $R$.
Since $Q^x$ fixes $\tau$, then $Q$ fixes $\tau^{x^{-1}}$.
Now
\[\epsilon^{y^{-1}}=(\tau^{gt})^{y^{-1}}=\tau^{x^{-1}xgty^{-1}}=\tau^{x^{-1}qv}=\tau^{x^{-1}v} \, .\]
So $\epsilon$ and $\tau$ are $\norm GR$-conjugate, and thus they are equal.

Now we define a map
$$f: \bigcup_{Q \in {\mathcal R}}\ipi{G|Q, \varphi}
\rightarrow 
\bigcup_ {\tau \in {\c A}}\left(\left( \bigcup_{Q \in {\mathcal Q}_\tau}\ipi{G_\tau|Q, \tau}\right) \times \{ \tau\}\right) $$
given by
$f(\mu)=(\mu_\tau, \tau)$, 
where $\tau \in \c A$ is the unique element in $\c A$
such that $\mu_\tau \in \ipi{G_\tau|Q^x, \tau}$, for some $x \in \norm GR$. Since $\mu_\tau^G=\mu$, we have that $f$ is injective.
If we have that $\gamma \in \bigcup_{Q \in {\mathcal Q}_\tau}\ipi{G_\tau|Q, \tau}$ then $\gamma^G
\in \bigcup_{Q \in {\mathcal Q}}\ipi{G|Q, \varphi}$, so $f$ is surjective. 
\end{proof}

Some of the difficulties in Theorem~\ref{thm} are caused by the fact that Clifford correspondence does not necessarily respect
vertices, even in quite restricted situations. Suppose that $N$ is a normal $p'$-subgroup of $G$, $\tau \in \irr N$, $Q$
is a $p$-subgroup of $G$ and $\tau$ is $Q$-invariant. Then
it is not necessarily true that induction defines a bijection
$\ibr{G_\tau|Q,\tau} \rightarrow \ibr{G|Q,\tau}$.
For instance, take $p=2$ and $G={\tt SmallGroup}(216,158)$.
This group has a unique normal subgroup 
$N$ of order 3. The Fitting subgroup $F$ of $G$ is $F=N \times M$,
where $M$ is a normal subgroup of type $C_3 \times C_3$,
and $G/F=D_8$.  Let $1 \ne \tau \in \irr N$. Then $G_\tau \nor G$
has index 2, and $G_\tau /N=S_3 \times S_3$.
Now $\tau$ has a unique extension $\hat\tau \in \ibr{G_\tau}$.
The group $G_\tau$ has three conjugacy classes of subgroups $Q$ of order 2.
Take $Q_1$ that corresponds to $C_2 \times 1$ and $Q_2$ that corresponds
to $1 \times C_2$. Then $Q_1$ and $Q_2$ are not $G_\tau$-conjugate
but $G$-conjugate. So $|\ibr{G_\tau|Q_1,\tau}|=1$ and 
$|\ibr{G|Q_1,\tau}|=2$.

\section{A canonical bijection}

\medskip   
If $G$  has a normal Hall $\pi$-subgroup,
then we have a canonical bijection in Theorem~\ref{thm}. This seems worth to be explored.

\medskip 

\begin{lem}\label{A}
Suppose that $G=NH$ where $N$ is a normal $\pi$-subgroup and $H$ is a $\pi'$-subgroup. Then $\norm GQ=\cent NQ\norm HQ$ for every $Q\le H$.
\end{lem}

\medskip

\begin{proof}
First note that
\[Q=Q(N\cap H)=QN\cap H\nor N\norm GQ\cap H\le\norm HQ.\]
Let $xh\in\norm GQ$ where $x\in N$ and $h\in H$. Then $h=x^{-1}(xh)\in N\norm GQ\cap H\le\norm HQ$. This shows $\norm GQ=\norm NQ\norm HQ=\cent NQ\norm HQ$.
\end{proof}

\medskip

\begin{lem}\label{B}
Suppose that $G=NH$ where $N$ is a normal $\pi$-subgroup and $H$ is a solvable $\pi'$-subgroup.
Let $R\le H$, and let $\tau \in \irr{\cent NR}$ be such that 
$\norm GR_\tau=\cent NR \times R$.
Let $\gamma \in {\rm Irr}_R(N)$ be the Glauberman
correspondent of $\tau$. Then  $R=\norm {H_\gamma}R$.
\end{lem}

\medskip

\begin{proof}
Suppose that $R<S \le H_\gamma$, where $R \nor S$.
Then $S$ acts on the $R$-Glauberman correspondence.
Since $S$ fixes $\gamma$, therefore it fixes $\gamma^*=\tau$.
But this gives the contradiction 
$S \sbs \norm GR_\tau=\cent NR \times R$.
\end{proof}

\medskip

\begin{thm}\label{C}
Suppose that $G=NH$ where $N$ is a normal $\pi$-subgroup and $H$ is a solvable $\pi'$-subgroup.
Let $R$ be a nilpotent subgroup 
of $H$.
Let $\mathcal Q$ be the set of subgroups $Q \sbs H$ 
such that $R$
is a Carter subgroup of $Q$.
Then there is a natural bijection
$$\bigcup_{Q \in {\mathcal Q}}\ipi{G|Q} \rightarrow \ipi{\norm GR|R} \,.$$
\end{thm}

\medskip

\begin{proof}
Let $Q \in \mathcal Q$.
By the Frattini argument, notice
that $\norm GQ =Q(\norm GQ \cap \norm GR)$,
and that $Q \cap (\norm GQ \cap \norm GR)=R$.

Let $\phi \in \ipi{G|Q}$. By Lemma~\ref{gvert}, there exists a $Q$-invariant $\theta \in \irr N$
under $\phi$. Then $T=G_\theta=QN$ using Corollary~8.16 in \cite{I1} for instance. 
If $\theta_1$ is another such choice, then $\theta_1=\theta^g$ for some
$g \in \norm GQ$. 
Thus, we may assume that $g \in \norm GQ \cap \norm GR$.
Let $\theta^* \in \cent NR$ be the $R$-Glauberman correspondent
of $\theta$. Now, by Lemma \ref{A} applied in $T$,
we have that $\norm TR=\cent NR \norm QR=\cent NR \times R$. 
We claim that $\norm TR$ is the stabilizer of $\theta^*$ in $\norm GR$.
If $x \in \norm GR$ fixes $\theta^*$, then $x$ fixes $\theta$, and thus $x \in \norm TR$,
as claimed.  Now $\phi^*:=(\theta^*\times 1_R)^{\norm GR}$ is irreducible,
and belongs to $\ipi{\norm GR|R}$. Since $\theta_1$ is $\norm GQ \cap \norm GR$-conjugate to $\theta$, $\phi^*$ is independent of the choice of $\theta$. 

Suppose that $\phi^*=\mu^*$, where $\phi \in \ipi{G|Q_1}$
and $\mu \in \ipi{G|Q_2}$, where $R$ is a Carter subgroup of $Q_i$
and $Q_i \sbs H$.  Suppose that we picked $\theta$
for $\phi$ and $\epsilon$ for $\mu$, so that
$\phi^*=(\theta^* \times 1_R)^{\norm GR}$ and 
$\mu^*=(\epsilon^* \times 1_R)^{\norm GR}$.
Then $\norm GR= \cent NR\norm HR$, and $\theta^*$ and $\epsilon^*$
are $\norm HR$-conjugate, say $(\theta^*)^x=\epsilon^*$.  Then $\theta^x=\epsilon$.
By replacing $(Q_1,\theta)$ by $(Q_1^x,\theta^x)$, we may assume that
$\theta=\epsilon$.  But then $Q_1=H_\theta=H_\epsilon=Q_2$. Since $\pi$-partial character are determined on the $\pi$-elements, we must have $\phi=\mu$ now.

Suppose conversely that  $\tau \in \ipi{\norm GR|R}$. Then $\tau$ 
is induced from $\cent NR \times R$. 
Let $\mu \in \irr{\cent NR}$ such that $\mu \times 1_R$ induces
$\tau$. Then the stabilizer of $\mu$ in $\norm GR$ is
$\cent NR \times R$. If $\rho \in {\rm Irr}_R(N)$ is the $R$-Glauberman correspondent
of $\mu$, then by Lemma~\ref{B} we know that $R$ is a Carter subgroup
of $Q=H_\rho$, where $Q$ is the stabilizer in $H$ of $\rho$. 
Thus with the notation of the first part of the proof we obtain $\tau=\mu^*$ where $\mu$ is induced from $G_\rho=QN$.
\end{proof}

\end{document}